\documentclass[3p,10pt,a4paper,twoside,fleqn,sort&compress]{filomat}
\usepackage{amssymb,amsmath,latexsym}

\DeclareMathOperator{\RE}{Re} 
\newtheorem{theorem}{Theorem}[section]

\newtheorem{corollary}[theorem]{Corollary}

\newtheorem{lemma}[theorem]{Lemma}

\newcommand{\FilVol}{xx}
\newcommand{\FilYear}{yyyy}
\newcommand{\FilPages}{zzz--zzz}
\newcommand{\FilDOI}{}
\newcommand{\FilReceived}{Received: ; Accepted:   }

\begin{document}

\title{Estimates for   Coefficients  of  Certain  Analytic Functions
\footnote{Dedicated to Prof.   Ciric.}}


\author[affil1]{V. Ravichandran}
\ead{vravi68@gmail.com; vravi@maths.du.ac.in}


\author[affil1] {Shelly Verma}
\ead{jmdsv.maths@gmail.com}


\address[affil1]{Department of Mathematics, University of Delhi, Delhi--110 007, India}

\newcommand{\AuthorNames}{ V. Ravichandran and S. Verma}


\newcommand{\FilMSC}{Primary 30C45; Secondary 30C50, 30C80}
\newcommand{\FilKeywords}{Univalent functions, starlike functions, convex functions,  meromorphic functions, coefficient bounds, inverse coefficient bounds}
\newcommand{\FilCommunicated}{}
\newcommand{\FilSupport}{The second  author is supported by a Senior Research Fellowship from  National Board for Higher Mathematics, Mumbai.}


\begin{abstract}
For $ -1 \leq B \leq 1$ and $A>B$, let $\mathcal{S}^*[A,B]$ denote  the class of generalized Janowski starlike functions consisting  of all normalized  analytic  functions  $f$ defined by the subordination $z f'(z)/f(z) \prec (1+ A z)/(1+ B z)$ $(|z|<1)$. For $-1 \leq B \leq 1<A$,  we  investigate  the inverse coefficient problem for functions in the class $\mathcal{S}^*[A,B]$ and its meromorphic counter part. Also,  for $ -1 \leq B \leq  1 < A $, the sharp  bounds for first five  coefficients for inverse functions  of   generalized  Janowski convex functions are determined. A simple and  precise proof for inverse coefficient estimations for  generalized  Janowski convex functions  is provided for the case $A= 2 \beta -1$ $(\beta >1)$ and $B=1$.  As an application, for $F:=f^{-1}$, $A= 2 \beta -1$ $(\beta >1)$ and  $B=1$, the sharp coefficient bounds of $F/F'$ are obtained when  $f$ is a  generalized Janowski starlike or generalized Janowski convex function.  Further, we  provide  the sharp coefficient estimates for inverse functions  of normalized analytic  functions $f$ satisfying $f'(z) \prec (1+z)/(1+B z)$ $(|z|<1,  -1 \leq B < 1)$.
\end{abstract}

\maketitle

\section{Introduction and Preliminaries}

Let $\mathbb{D}$ denote the unit disc. Let $\mathcal{A}$ be the class of all normalized analytic  functions $f : \mathbb{D} \to \mathbb{C}$ of the form $f(z)= z+ a_2 z^2 + a_3 z^3 + \cdots$. The subclass of $\mathcal{A}$ consisting of univalent functions is denoted by $\mathcal{S}$.
An analytic function $f$ is said to be subordinate to an analytic function $g$, written  $ f \prec g$, if $f=g \circ w$ for some analytic function $w:\mathbb{D}\to \mathbb{D}$  with $w(0)=0$. If $g$ is univalent, then $f \prec g$ is equivalent to $f(0)=g(0)$ and  $f(\mathbb{D}) \subset g(\mathbb{D})$.  Let  $\varphi$ be an analytic univalent function with positive real part  mapping $\mathbb{D}$ onto domains symmetric with respect to real axis and starlike with respect to $\varphi(0)= 1$ and $\varphi'(0) >0$. Let $\mathcal{P}(\varphi)$ denote the class of all analytic functions $p: \mathbb{D} \to \mathbb{C}$  such that $p \prec \varphi$.  For such $\varphi$,   Ma and Minda \cite{MR1343506}  introduced the subclasses  $\mathcal{S}^*{(\varphi)}$ $(\mathcal{K}(\varphi))$  of $\mathcal{S}$   consisting of functions $f \in \mathcal{S}$ such that  $z f'(z)/ f(z)$ $ (1 + z f''(z)/f'(z))\in \mathcal{P}(\varphi).$
For different choices of $\varphi$, several well-known classes can be  easily  obtained  from these classes which were earlier considered and studied one by one  for their  geometric and  analytic properties. For instance,  $\mathcal{S}^*{((1+z)/(1-z))}=:\mathcal{S}^*$ and  $\mathcal{K}((1+z)/(1-z))=:\mathcal{K}$, the usual classes of starlike and convex functions respectively;  for $0 \leq \alpha <1$, $\mathcal{S}^*{((1+(1-2 \alpha)z)/(1-z))}=:\mathcal{S}^*(\alpha)$ and  $\mathcal{K}((1+(1-2 \alpha)z)/(1-z))=:\mathcal{K}(\alpha)$, the well-known classes of starlike and convex  functions of order $\alpha$,  respectively introduced in \cite{MR1503286}; for $0 < \alpha \leq 1$,   $\mathcal{S}^*{(((1+z)/(1-z))^{\alpha})}=:\mathcal{SS}^*(\alpha)$ is   the well-known class of  strongly starlike functions of order $\alpha$ introduced in  \cite{MR0251208}.
  In \cite{MR1343506}, the authors gave  a unified treatment to the geometric as well as analytic properties of these well-known classes.

We observe that the distortion theorem,  upper bound of $|f|$, rotation theorem, upper bound of Feketo-Szeg\"{o} coefficient functional $|a_3 - \mu a_2^2|$ for  $f \in \mathcal{K}(\varphi)$ given in \cite{MR1343506}  still hold  for a  normalized locally univalent function $f$   satisfying $1+ z f''(z)/f'(z) \prec \varphi(z)$ if we drop the condition that $\varphi$ has positive real part. Consequently,  the growth theorem and upper bound of Feketo-Szeg\"{o} coefficient functional $|a_3 - \mu a_2^2|$ follow for a  normalized analytic  function $f$  satisfying $ zf'(z)/f(z) \prec \varphi(z)$ even if $\varphi$  does not have  positive real part.  This motivates one  to consider the following subclasses  of $\mathcal{A}$,  for $-1 \leq B \leq  1$, $A>B$,
 \begin{align*}
\mathcal{K}[A,B] =  \left\lbrace f \in \mathcal{A}:1 + \frac{z f''(z)}{f'(z)} \in \mathcal{P}[A,B] \right\rbrace
\quad  \text{and } \quad \mathcal{S}^*[A,B]=  \left\lbrace f  \in \mathcal{A}: \frac{ z f'(z)}{f(z) } \in \mathcal{P}[A,B]  \right\rbrace
\end{align*}
  where $ \mathcal{P}[A,B]:= \mathcal{P}((1+ Az)/(1+ Bz))$.  For  $-1 \leq B<A \leq 1$, $\mathcal{S}^*[A,B]$ is a  subclass of $\mathcal{S}^*$ introduced by Janowski \cite{MR0328059} and   for  particular values of $A$ and $B$, it  reduces to  several known subclasses of $\mathcal{S}^*$. Precisely, $\mathcal{S}^*[1-2 \alpha,-1]=: \mathcal{S}^*(\alpha)$ $(0 \leq \alpha <1)$\cite{MR1503286}; \/  $\mathcal{S}^*[1,1/M-1]=: \mathcal{S}^*(M)$ $(M >1/2)$\cite{MR0267103}; \/ $\mathcal{S}^*[\beta,-\beta]=: \mathcal{S}^{*^{(\beta)}}$ $(0 < \beta \leq 1 )$ \cite{MR0241626}; \/  $\mathcal{S}^*[1- \beta,0]=: \mathcal{S}^{*}_{1-\beta}$  $( 0 \leq \beta < 1)$ \cite{MR509608}.  Note that,  for  $-1 \leq B  \leq  1 <  A$, the  functions in the classes  $\mathcal{K}[A,B]$ and  $\mathcal{S}^*[A,B]$ may not be univalent but  must be  locally univalent  in $\mathbb{D}$  and non-vanishing in  $\mathbb{D}\setminus \{0\}$, respectively.


Recently, the classes $\mathcal{S}^*[2 \beta-1,1]$ and  $ \mathcal{K}[2 \beta-1,1]$ $(\beta >1)$ have been studied by several authors, see \cite{MR1896244, MR1917801, MR1304483}. Moreover, the upper bound of the Feketo-Szeg\"{o} coefficient functional $|a_3 - \mu a_2^2|$   for    $ f \in  \mathcal{K}[2 \beta-1,1]$ or $f \in  \mathcal{S}^*[2 \beta-1,1]$;  the  distortion theorem, upper bound of $|f|$,  rotation theorem for    $ f \in  \mathcal{K}[2 \beta-1,1]$; and the growth theorem for  $ f \in  \mathcal{S}^*[2 \beta-1,1]$   are given  in  \cite{MR3436767}  which  can  actually  be deduced, even for the  functions in the generalized classes  $\mathcal{S}^*[A,B]$ and  $\mathcal{K}[A,B]$ $(-1 \leq B \leq 1<A)$,  from  the  results in \cite{MR1343506}.
 Also, for $-1 \leq B \leq  1$ and $A>B$,  one can  consider   the meromorphic counter part of  $\mathcal{S}^*[A,B]$, namely, the   class  $\Sigma^*[A,B]$ consisting of  analytic  functions  of the form
\begin{align}
g(z)= z + b_0 +  \dfrac{b_1}{z} + \dfrac{b_2}{z^2} + \cdots\label{p3eq0.01}
\end{align}
defined on $\mathbb{C} \setminus \overline{\mathbb{D}}$ such that  $ z g'(z)/g(z) \in p_0(\mathbb{D})$ where $p_0: \mathbb{D} \to \mathbb{C}$ is defined by $p_0(z)=(1+Az)/(1+Bz)$. For $-1 \leq B < A \leq 1$, the class $\Sigma^*[A,B]$  has been considered in \cite{MR766797}  and the  particular choices of $A$ and $B$ give the meromorphic counter parts of the classes corresponding to those of $S^*[A,B]$  such as   $\Sigma^*[1-2 \alpha,-1]=: \Sigma^*(\alpha)$ $(0 \leq \alpha <1)$ \cite{MR0150279}; \/  $\Sigma^*[1,1/M-1]=: \Sigma^*(M)$ $(M >1/2)$\cite{MR0286994}; \/ $\Sigma^*[\beta,-\beta]=: \Sigma^{*^{(\beta)}}$ $(0 < \beta \leq 1 )$\cite{MR0241626}; \/  $\Sigma^*[1- \beta,0]=: \Sigma^{*}_{1-\beta}$  $( 0 \leq \beta < 1)$.
Hallenbeck \cite{MR0338338} introduced the class $\mathfrak{I}$ consisting of functions $f \in \mathcal{S}$ such that $f' \in \mathcal{P}$, where $\mathcal{P}:= \mathcal{P}\big((1+z)/(1-z)\big)$. Further, Libera  and Z{\l}otkiewicz \cite{MR681830, MR749890} investigated the inverse coefficient problem of functions in the class $\mathfrak{I}$. For $-1 \leq B < A \leq 1$, let $\mathfrak{I}[A,B]$ denote the subclass of $\mathcal{S}$ consisting of functions $f \in \mathcal{S}$ such that $f' \in \mathcal{P}[A,B]$.

The problem of estimating the coefficients of inverse functions lay its origin in 1923 when L\"{o}wner \cite{MR1512136} gave   the sharp coefficient estimates for inverse  function  of  $f \in \mathcal{S}$ along with the sharp coefficient estimation for the third coefficient of $f \in \mathcal{S}$. Later, several authors \cite{MR0188428, MR589658, MR0335777, MR0011721}  gave alternate proofs for the inverse coefficient problem for functions in the class $\mathcal{S}$  but the inverse coefficient problem is still an open problem even  for  the  well-known classes   $\mathcal{K}$ and $\mathcal{S}^*(\alpha)$ $(0 \leq \alpha <1)$, although the  sharp estimates for  initial inverse coefficients are known for these classes, for details see \cite{MR689590, MR652447, MR2296897}. This leads to  several works  related to the inverse coefficient problem for functions in certain subclasses of $\mathcal{S}$, see \cite{MR2257293, MR2868315, MR813267, MR737480, MR1140278,  MR763927, MR1040905,MR2055766}. Recently, the inverse coefficient problem is  completely settled in \cite{MR3436767} for functions  in the classes $\mathcal{S}^*[2 \beta-1,1]$ or $\Sigma^*[2 \beta-1,1]$ or  $\mathcal{K}[2 \beta-1,1]$, $\beta >1$.

  In this paper, we  are mainly concerned about the determination of the sharp  inverse coefficient bounds for functions in  the classes $\mathcal{S}^*[A,B]$  or $\Sigma^*[A,B]$ $(-1 \leq B \leq 1<A)$. Also, we are giving  the  sharp  coefficient bounds for the inverse functions of functions in the class $\mathfrak{I}[1,B]$ $(-1 \leq B <1)$ and the  sharp  first five  coefficient bounds for the inverse functions of functions in the class $\mathcal{K}[A,B]$ for $ -1 \leq B \leq  1 < A $.  Apart from this,  we  present a slightly simpler proof than  the proof given in \cite{MR3436767}  for the sharp inverse coefficient estimation for  functions  in the class $\mathcal{K}[2 \beta-1,1]$ $(\beta >1)$. As an application, for $F:=f^{-1}$ and $\beta >1$, the sharp coefficient bounds of $F/F'$ are obtained when  $f \in \mathcal{S}^*[2 \beta-1,1]$ or $f \in \mathcal{K}[2 \beta-1,1]$.
 Further,   under some conditions,   the sharp coefficient estimates are determined  for functions in the class $\Sigma^*[A,B]$ $(-1 \leq B \leq 1< A)$.

  We need the following lemmas to prove our results.
  \begin{lemma}\cite[Theorem II, p.\ 547]{MR0059359}\label{p3lem1}  Let $\Omega$ be the family of functions $f$ such that for $|z| < \rho$ with $\rho>0$ , $f(z)= \sum_{n=1}^{\infty} a_n z^n$ $(a_1 \neq 0)$. If $ f \in \Omega$  and $\phi$ is the inverse function of $f$, then $\phi \in \Omega$. For any integer $t$, let  $f(z)^t = \sum_{n=-\infty}^{\infty} a_n^{(t)} z^n$ and $\phi(w)^t= \sum_{n=-\infty}^{\infty} b_n^{(t)} w^n$ in some neighbourhoods of the origin, where $a_n^{(t)}$ and $b_n^{(t)}$ are zero for $n < t$. Then
  $$ b_n^{(t)}= \dfrac{t}{n} a_{-t}^{(-n)}, \quad n \neq 0.$$
  For $n=0$, $b_0^{(t)}$ is defined by $$\sum_{t=-\infty}^{\infty} b_0^{(t)} z^{-t-1} = \dfrac{f'(z)}{f(z)}.$$
  \end{lemma}
  \begin{lemma}\cite[Theorem X, p.\ 70]{MR0008625}\label{p3lem2}
  Let $f(z) = 1 + \sum_{n=1}^{\infty}a_n z^n$ and $g(z) =1+ \sum_{n=1}^{\infty}b_n z^n$ $(z \in \mathbb{D})$ be such that  $f \prec g$. If $g$ is univalent in $\mathbb{D}$ and $g(\mathbb{D})$ is convex,  then $|a_n| \leq |b_1|$.
  \end{lemma}
  By using the  above lemma, the following result is proved. This has been proved  in \cite{MR907789} for the case  $-1 \leq B < A \leq 1$.
  \begin{lemma}\label{p3lem3}
  If $p(z)= 1+ \sum_{k=1}^{\infty} c_k z^k$ is in $\mathcal{P}[A,B]$ $(-1 \leq B \leq 1, A>B)$ then $|c_n| \leq A-B$. The bounds are sharp.
  \end{lemma}
  \begin{proof}
 Since $p \in \mathcal{P}[A,B]$,  $p(z) \prec (1+ Az)/(1+Bz)$. Let $g(z):=(1+A z)/(1+ B z)$. Clearly, $g$ is univalent in $\mathbb{D}$. For $-1 <B <1$, $g(\mathbb{D})$ is the disc $|w-(1-AB)/(1-B^2)|<(A-B)/(1-B^2)$. For $B=1$ and $B=-1$,  $g(\mathbb{D})$ is the left half plane $\RE(w)<(1+A)/2$ and the right half plane $\RE(w)>(1-A)/2$ respectively. Therefore, $g(\mathbb{D})$ is convex and hence by Lemma \ref{p3lem2},   $| c_n| \leq  A-B$ for each $n$. Define  a function $p_n:\mathbb{D} \to \mathbb{C}$ as $$p_n(z)=\dfrac{1+ A z^n}{1+ B z^n}= 1 + (A-B) z^n - B(A-B) z^{2n} + \cdots.$$Clearly, the result is sharp for the function $p_n$.
  \end{proof}
  The following lemma follows easily by induction on $m$ and for  $-1 \leq B < A \leq 1$, it is given in \cite[Lemma 2,  p.\ 737]{MR907789}.
  \begin{lemma}\label{p3lem4}
  Let $A>B$, $-1 \leq B \leq 1$. Then for any integer $t$ and   $m  \in \mathbb{N}$, we have
 $$ m^2 \prod_{j=0}^{m-1} \left(\frac{(A-B) t + B j}{j+1} \right)^2  = (A-B)^2 t^2 + \sum_{k=1}^{m-1} \Big(\big( (A-B) t + B k\big)^2-k^2 \Big)  \prod_{j=0}^{k-1} \left(\frac{(A-B) t + B j}{j+1} \right)^2.
 $$
\end{lemma}
\section{Main Results}
The following theorem gives estimates for inverse coefficient of functions in the class $\mathcal{S}^*[A,B]$  $(  -1 \leq B  \leq 1<A)$.
\begin{theorem}\label{p3thm1}
Let $f \in \mathcal{S}^*[A,B]$  $( -1 \leq B  \leq 1 <A)$ and  $f^{-1}(w)=:F(w)= w + \sum_{n=2}^\infty \gamma_{n} w^n$ in some neighbourhood of the origin.  Then  for each $n \geq 2$,  \begin{align} |\gamma_{n}| \leq  \dfrac{1}{n}\prod_{m=0}^{n-2} \left(\dfrac{ n(A -B) + m B}{m+1} \right). \label{p3eq2}
\end{align}
The result is sharp.
\end{theorem}
\begin{proof}
For  any integer $t>0$, let $$g(z):= \left(\dfrac{f(z)}{z}\right)^{-t}= 1 + \sum_{j=1}^{\infty} a_j^{(-t)} z^j \quad (|z|<1).$$ Then
\begin{align}
 -\dfrac{z}{t}\dfrac{g'(z)}{g(z)}= \dfrac{z f'(z)}{f(z)}-1.\label{p3eq5}
  \end{align}
  Since $f \in \mathcal{S}^*[A,B]$, we have
 \begin{align}
 \dfrac{z f'(z)}{f(z)}= \dfrac{1+ A w(z)}{1+ B w(z)}\label{p3eq6}
 \end{align}
 for some analytic function  $w: \mathbb{D} \to \mathbb{D}$ with $w(0)=0$. The equations \eqref{p3eq5} and \eqref{p3eq6} give  $$ \sum_{j=1}^\infty  j a_{j}^{(-t)} z^{j} = - w(z) \left( (A-B) t + \sum_{j=1}^\infty \big(B(j-t) + A t \big)a_{ j}^{(-t)} z^{j}\right) $$  which can be rewritten as
 $$  \sum_{j=1}^s  j a_{j}^{(-t)} z^{j}  + \sum_{j=s+1}^\infty b_{ j}^{(-t)} z^{j}= -w(z) \left( (A-B) t+ \sum_{j=1}^{s-1} \big(B(j-t) +  A t \big)a_{ j}^{(-t)} z^{j}\right)$$
  where   $$ \sum_{j=s+1}^\infty b_{j}^{(-t)} z^{j}:= \sum_{j=s+1}^{\infty}  j a_{j}^{(-t)} z^{j} + w(z)\left(\sum_{j=s}^{\infty} \big(B(j-t) +  A t \big)a_{j}^{(-t)} z^{j}\right). $$
Since $|w(z)| <1$ $(|z| <1)$, squaring the moduli of both sides, we have $$ \left|\sum_{j=1}^s  j a_{j}^{(-t)} z^{j}  + \sum_{j=s+1}^\infty b_{j}^{(-t)} z^{j}\right|^2 < \left|  (A-B) t + \sum_{j=1}^{s-1} \big(B(j-t) +  A t \big)a_{j}^{(-t)} z^{j}\right|^2.  $$ Integrating   along $|z|=r$, $0<r<1$ with respect to  $\theta$ $(0 \leq \theta \leq 2 \pi)$ and applying Parseval's identity that for an analytic function $g: \mathbb{D} \to \mathbb{C}$ of the form $g(z)= \sum_{n=0}^{\infty} A_n z^n$,    $$ \frac{1}{2 \pi} \int_0^{2 \pi} |g(r e^{i \theta})|^2 \, d\theta = \sum_{n=0}^{\infty} |A_n|^2 r^{2n}  \quad (0 < r <1)$$  we have
$$ \sum_{j=1}^s  |j a_{j}^{(-t)}|^2 r^{2j}  + \sum_{j=s+1}^\infty| b_{j}^{(-t)}|^2 r^{2j} \leq  (A-B)^2 t^2 +  \sum_{j=1}^{s-1} |B(j-t) + A t|^2 |a_{j}^{(-t)}|^2 r^{2j}. $$
Letting $r \to 1$ yields
$$ \sum_{j=1}^s  |j a_{j}^{(-t)}|^2    \leq  (A-B)^2 t^2 + \sum_{j=1}^{s-1} |B(j-t) + A t|^2 |a_{j}^{(-t)}|^2$$
and therefore,
 \begin{align}
 |s a_{s}^{(-t)}|^2    \leq  (A-B)^2 t^2 + \sum_{j=1}^{s-1} \Big( \big((A-B) t + B j\big)^2 - j^2\Big)|a_{j}^{(-t)}|^2.\label{p3eq7}
\end{align}
We shall show that,   for $-1 \leq B \leq 1, \, A>B$, $t \geq ( s-1)(1-B)/(A-B)$ and $s \geq 1$,
\begin{align}
  | a_{s}^{(-t)}|  \leq  \prod_{m=0}^{s-1} \left( \dfrac{(A-B) t + m B }{m+1}\right).\label{p3eq8}
 \end{align} We proceed by induction on $s$.
 For  $s=1$, equation \eqref{p3eq7} gives
  $$  | a_{1}^{(-t)}|    \leq  (A-B)t.$$
 Since $-1 \leq B \leq1$ and $A>B$, for  fixed $j \geq 1$,  $\big((A-B) t + B j\big)^2 - j^2= \big((A-B)t-j(1-B)\big)\big((A-B) t + j(1+B)\big) \geq 0$ if $t \geq j(1-B)/(A-B)$. Assume that \eqref{p3eq8} holds for  $s \leq q-1$ and  $ t  \geq ( q-1)(1-B)/(A-B)$. Then by using induction hypothesis and the equation  \eqref{p3eq7} for $s=q$,    we have
$$ |q a_{ q}^{(-t)}|^2    \leq  (A-B)^2 t^2 + \sum_{j=1}^{q-1} \Big( \big((A-B) t + B j\big)^2 - j^2\Big) \prod_{m=0}^{j-1} \left( \dfrac{(A-B) t + m B }{m+1}\right)^2 $$
which by   using Lemma \ref{p3lem4}  gives
$$ | a_{q}^{(-t)}|   \leq    \prod_{m=0}^{q-1}  \left(\dfrac{(A-B) t + m B }{m+1}\right).$$
Thus,  \eqref{p3eq8} holds for  $s=q$ and hence by induction \eqref{p3eq8}  holds for all $s \geq 1$. By applying Cauchy's integral formula for $F'$, it can be easily seen that
\begin{align}
\gamma_{n}= \dfrac{1}{n} a_{n-1}^{(-n)}\quad  (n \geq 2).\label{p3eq8.1}
\end{align}
Since $A > 1$,  therefore $( n-2)(1-B)/(A-B) \leq  n-2$ $(n \geq 2)$. So, for $t=n$ and $s= n-1$,   the equation  \eqref{p3eq8} gives
 $$ |\gamma_{n}| = \dfrac{1}{n} |a_{n-1}^{(-n)}| \leq  \dfrac{1}{n} \prod_{m=0}^{n-2} \left( \dfrac{(A-B) n + m B }{m+1}\right). $$
Define a  function $f_{1}: \mathbb{D} \to \mathbb{C}$  by
\begin{align}
f_{1}(z)=\begin{cases}
z(1+ B z)^{(A-B)/ B}, &\text{$B \neq 0$}\\
 ze^{A z}, &\text{$B=0$.}
\end{cases}  \label{p3eq11.4}
\end{align}
The  result is sharp   for the function $f_1$.
\end{proof}
For  $A=2 \beta-1$, $B=1$ $(\beta >1)$, the above theorem reduces to  \cite[Theorem 4.3, \ p. 14]{MR3436767}.
\begin{corollary}
Let $f \in \mathcal{S}^*[2 \beta-1,1]$ $(\beta>1)$  and  $f^{-1}(w)=:F(w)= w + \sum_{n=2}^\infty \gamma_{n} w^n$ in some neighbourhood of the origin. If $F(w)/F'(w)= w + \sum_{n=2}^{\infty}\delta_n w^n$, then $|\delta_2| \leq 2(\beta-1)$ and  for $n>2$,
$$ |\delta_n| \leq 2(\beta-1) \prod_{j=2}^{n-1}\left(\dfrac{2(n-1) (\beta-1)+ j}{j}\right).$$
The result is sharp.
\end{corollary}
\begin{proof}
Since $f \in \mathcal{S}^*[2 \beta-1,1]$ $(\beta >1)$, $z f'(z)/f(z) \in \mathcal{P}[2 \beta-1,1]$. This gives  $$ \dfrac{z f'(z)}{f(z)}= p(z)$$
where $p(z)= 1 + c_1 z + c_2 z^2 + \cdots \in \mathcal{P}[2 \beta-1,1]$. In terms of $F:=f^{-1}$, the above equation becomes
$$ \dfrac{F(w)}{F'(w)}= w p(F(w)).$$
Using power series expansions of $F/F'$, $p$ and $F$, we obtain
\begin{align}
\sum_{n=2}^{\infty} \delta_n w^n = \sum_{n=2}^{\infty} \left( \sum_{j=1}^{n-1} c_j \gamma_{n-1,j} \right) w^n \label{p3eq11.3555}
\end{align}
 where $\gamma_{n-1,j}$ denotes the coefficient of $w^{n-1}$ in the expansion of $F(w)^j$. In fact, $\gamma_{n-1,j}= S_j(\gamma_{2}, \gamma_{3}, \dotsc, \gamma_{n-2})$ is a polynomial in $\gamma_{2}, \gamma_{3}, \dotsc, \gamma_{n-2}$ with non-negative coefficients and  $\gamma_{n-1,n-1}=1$. On comparing the coefficients of $w^n$, we have
 $$ \delta_n=  \sum_{j=1}^{n-1} c_j \gamma_{n-1,j}.$$
 An application of Lemma \ref{p3lem3}  gives
 \begin{align}
 |\delta_n| \leq  2(\beta-1) \sum_{j=1}^{n-1}  S_j(|\gamma_{2}|, |\gamma_{3}|, \dotsc, |\gamma_{n-2}|). \label{p3eq11.355}
 \end{align}
 Define  $g_1(z):= e^{-i \pi}f_1(e^{i \pi}z)$ where $f_1$ is given by \eqref{p3eq11.4} for $A=2 \beta-1$ and $B=1$. Clearly, $g_1 \in \mathcal{S}^*[2 \beta-1,1]$. Then  $G_1(w):=g_1^{-1}(w)= w +\sum_{n=2}^{\infty} A_n w^n$ and $G_1(w)/G_1'(w)= w - \sum_{n=2}^{\infty} B_n w^n$ where $w$ lies in some neighbourhood of the origin,
 \begin{align*}
B_2&:= 2(\beta-1), \quad B_n:=  2(\beta-1) \prod_{j=2}^{n-1}\left(\dfrac{2(n-1) (\beta-1)+ j}{j}\right) \quad (n>2)
\intertext{and}
 A_n&:= \dfrac{1}{n}\prod_{m=0}^{n-2} \left(\dfrac{ 2n(\beta-1) + m }{m+1} \right)\, \, (n \geq 2).
\end{align*}
Proceeding as in \eqref{p3eq11.3555} for $g_1$ and then comparing the coefficients  of $w^n$  give
\begin{align}
B_n=2(\beta-1) \sum_{j=1}^{n-1}  S_j(A_{2},A_{3}, \ldots, A_{n-2})\,\,(n \geq 2).\label{p3eq11.3502}
\end{align}
Since $f \in \mathcal{S}^*[2 \beta-1,1]$, applying Theorem \ref{p3thm1} in \eqref{p3eq11.355} and using \eqref{p3eq11.3502} give  $|\delta_n| \leq B_n$.  Clearly, the sharpness follows for the function  $g_1$.
\end{proof}

\begin{corollary}\label{p3thm1.01}
Let $g $,  given by \eqref{p3eq0.01},  be in $\Sigma^*[A,B]$ $( -1 \leq B \leq 1<A)$ and  $n(1-B)-(A-B) \leq 0$. Then for each $n \geq 0$,
$$ |b_n| \leq  \prod_{m=0}^n \left(\dfrac{(A-B) + m B}{m+1}\right).$$
The result is sharp.
\end{corollary}
\begin{proof}
It is easy to observe that for any $g \in \Sigma^*[A,B]$, there exists $ f \in \mathcal{S}^*[A,B]$ such that for $z \in \mathbb{C} \setminus \overline{\mathbb{D}}$,  $g(z) = 1/f(1/z)$. Also, we note that the expansions  of  $f(z)^{-1}$  about the origin and $f(1/z)^{-1}$ about the infinity have same coefficients. Thus,  if  $z/f(z) = 1 + \sum_{n=1}^{\infty} a_n^{(-1)} z^n$  $(z \in \mathbb{D})$, then for $z \in \mathbb{C} \setminus \overline{\mathbb{D}}$,  we have
$$ \dfrac{g(z)}{z}= \dfrac{ 1}{zf( 1/z) }= 1 + \sum_{n=1}^{\infty} a_n^{(-1)} z^{-n}.$$
On comparing the coefficients, we obtain
\begin{align}
b_{n}= a_{n+1}^{(-1)} \quad (n \geq 0). \label{p3eq11.40}
\end{align}
An application of  \eqref{p3eq8} for $t=1$ and $s=n+1$ in the equation \eqref{p3eq11.40} gives the desired estimate.
Define  a function $
g_{1}: \mathbb{C} \setminus \overline{\mathbb{D}} \to \mathbb{C}$  by  \begin{align}
 g_{1}(z)= \frac{1}{f_{1}(1/z)} \label{p3eq11.21}
\end{align} where $f_{1}$ is given by \eqref{p3eq11.4}.
The  result is sharp for the function $g_1$  given by \eqref{p3eq11.21}.
\end{proof}
For $A=2 \beta-1$, $B=1$ $(\beta >1)$, the above result is mentioned in  \cite[Theorem 4.5, p.\ 17]{MR3436767}.
Next, we prove the  meromorphic counter part of the Theorem \ref{p3thm1}.
\begin{theorem}\label{p3thm2}
Let the function $g \in \Sigma^*[A,B]$ $(-1 \leq B  \leq 1<A)$ and $g^{-1}(w)= w + \sum_{n=0}^{\infty} \tilde{\gamma}_n w^{-n}$ in some neighbourhood of the infinity. Then  $|\tilde{\gamma}_0| \leq A-B$ and
$$ |\tilde{\gamma}_n| \leq  \dfrac{1}{n}\prod_{m=0}^n \left(\dfrac{(A-B)n + m B}{m+1}\right) \quad (n \geq 1).$$
 The result is sharp.
\end{theorem}
\begin{proof}
Since $g \in \Sigma^*[A,B]$, there exists $ f \in \mathcal{S}^*[A,B]$ such that for $z \in \mathbb{C} \setminus \overline{\mathbb{D}}$,  $g(z) = 1/f(1/z)$ and $g^{-1}(w)= 1/f^{-1}(1/w)$, see \cite[Theorem 2.4, p.\ 459]{MR0188428}. Therefore, for each $n \geq 0$,  \begin{align}
|\tilde{\gamma}_n| = |\gamma_{n+1}^{(-1)}| \label{p3eq12}
 \end{align}
 where $\gamma_{n+1}^{(-1)}$ is  the coefficient of $w^{-(n+1)}$ in  $1/(wf^{-1}(1/w))=  1+ \sum_{n=1}^{\infty} \gamma_{n}^{(-1)}w^{-n}$.

  Since $f \in \mathcal{S}^*[A,B]$,  we have $z f'(z)/f(z)= q(z) \in \mathcal{P}[A,B]$. If  $q(z) = 1+ \sum_{n=1}^{\infty} q_n z^n$, then by applying Lemma \ref{p3lem1}, we have
  $$  \sum_{p=-\infty}^{\infty} \gamma_1^{(p)} z^{-p-1} =  \dfrac{f'(z)}{f(z)}= \dfrac{q(z)}{z}= \dfrac{1}{z} \left(1+ \sum_{n=1}^{\infty} q_n z^n \right).$$
 Therefore, in view of \eqref{p3eq12} and Lemma \ref{p3lem3},  $|\tilde{\gamma_0}| = |\gamma_1^{-1}| = |q_1| \leq A-B.$
For $n \geq 1$, an application  of  Lemma \ref{p3lem1} and the inequality \eqref{p3eq8} for $t=n$, $s=n+1$  in  \eqref{p3eq12} gives
 $$ |\tilde{\gamma}_n| = |\gamma_{n+1}^{(-1)}| = \dfrac{1}{n}|a_{n+1}^{(-n)}| \leq \dfrac{1}{n} \prod_{m=0}^{n} \left( \dfrac{(A-B)n + m B }{m+1}\right).$$
 The sharpness follows for the function $g_1$ given by \eqref{p3eq11.21}.
\end{proof}
For $A=2 \beta-1$, $B=1$ $(\beta >1)$, the above theorem reduces to  \cite[Theorem 4.8, \ p. 18]{MR3436767}.
Recall that for $-1 \leq B < A \leq 1$, $$ \mathfrak{I}[A,B]:=\left\{f \in \mathcal{S}: f'(z) \prec \dfrac{1+Az}{1+Bz}\right\}.$$
The  following theorem  gives the sharp inverse coefficient estimates for functions in the class $\mathfrak{I}[1,B]$ and its proof  is based on the fact that if $p \in \mathcal{P}[A,B]$ $(-1 \leq B <A \leq 1)$, then $1/p \in \mathcal{P}[-B,-A]$ $(-1 \leq -A< -B \leq 1)$.
\begin{theorem}\label{p3thm3}
For $-1 \leq B <1$, let $f \in \mathfrak{I}[1,B]$  and $g(z)= \int_{0}^{z} (1-t)/(1- B t) \,dt $\/ $(|z|<1)$. If $f^{-1}(w)=:F(w)= w + \sum_{n=2}^\infty \gamma_{n} w^n$ and $g^{-1}(w)=:G(w) = w + \sum_{n=2}^\infty A_{n} w^n$ where $w$ lies in some neighbourhood of the origin,  then  for each $n \geq 2$,  $|\gamma_{n}| \leq  A_n$. The result is sharp.
\end{theorem}
\begin{proof}
Since $f' \in \mathcal{P}[1,B]$, $ f'(z) = p(z)$  for some  $p \in \mathcal{P}[1,B]$.  Let $w = f(z)$ then  $f'(z) F'(w)=1$  and so we have $$ F'(w)= P(F(w))$$  where    $P(z):= 1/p(z) = 1+ \sum_{n=1}^\infty c_n z^n  \in \mathcal{P}[-B,-1]$. This gives  $$1+ \sum_{n=1}^\infty (n+1) \gamma_{n+1} w^{n} =   1+ \sum_{n=1}^\infty c_n F(w)^n.$$ On comparing  the  coefficients of $w^n$, we have
\begin{align}  (n+1) \gamma_{n+1} = \sum_{i=1}^n c_i \gamma_{n,i} \quad (n \geq 1) \label{p3eq1}
\end{align}
 where $ \gamma_{n,i}$ denotes  the  coefficient of $w^n$ in the expansion of  $F(w)^i$ and $ \gamma_{n,n}=1$.
Since $g'(z) =   (1-z)/(1-B z) \in \mathcal{P}[1,B]$, proceeding as above,  we have
\begin{align}
G'(w)=  \dfrac{1- B G(w)}{1-G(w)} \label{p3eq1.111}
\end{align}
  which gives $$\sum_{n=1}^\infty (n+1) A_{n+1} w^{n} =   \sum_{n=1}^\infty (1-B)  G(w)^n.$$
Comparing  the  coefficients of $w^{n-1}$, we get
\begin{align}
n A_{n}=(1-B)  \sum_{i=1}^{n-1}  A_{n-1,i}  \quad (n  \geq 2) \label{p3eq1.112}
\end{align}
where $ A_{n-1,i}$ denotes  the  coefficient of $w^{n-1}$ in the expansion of  $G(w)^i$ and $ A_{n-1,n-1}=1$.
We first show that $A_n >0$ for all $n \geq 2$. By using the power series  expansion of $G$ in \eqref{p3eq1.111} and on comparing the coefficients of both sides, we obtain
\begin{align*}
2 A_2&= 1-B, \quad  3 A_3 = (1-B+2)A_2,  \quad \text{and} \\
(n+1)A_{n+1} &= (1-B+n)A_n + \sum_{k=1}^{n-2}(k+1) A_{k+1} A_{n-k} \quad (n > 2).
\end{align*}
Since $ -1 \leq B <1$, $A_2 = (1-B)/2 >0 $. By using induction on $n$, it can be easily seen from the above relations that $A_n >0$ for all $n \geq 2$.

 Next,  we shall show that  for all  $n \geq 2$, $|\gamma_n| \leq A_n$. We proceed by induction on $n$.
Since $P \in \mathcal{P}[-B,-1]$,  by using Lemma \ref{p3lem3},     $|c_i| \leq 1-B$  for each $i \geq 1$.  Clearly, the result holds for $n=2$.  Assume that  $ |\gamma_{i}| \leq A_i$ for $i \leq n-1$. It is easy to observe  that  $ \gamma_{n,i}= S_i(\gamma_{2}, \gamma_{3}, \dotsc, \gamma_{n-1})$ is a polynomial in $\gamma_{2}$, $\gamma_{3}$, $\dotsc$, $\gamma_{n-1}$ with non-negative coefficients and thus $ |\gamma_{n,i}| \leq  S_i(|\gamma_{2}|, |\gamma_{3}|, \dotsc, |\gamma_{n-1}|) \leq  S_i(A_2, A_3, \dotsc, A_{n-1})$.  Therefore, in view of \eqref{p3eq1} and \eqref{p3eq1.112},  we have $$ n |\gamma_{n}| \leq  \sum_{i=1}^{n-1} |c_i| |\gamma_{n-1,i}| \leq   (1-B) \sum_{i=1}^{n-1}   S_i(A_2, A_3, \dotsc, A_{n-2}) = (1-B)  \sum_{i=1}^{n-1}  A_{n-1,i} = n A_n$$ where $ A_{n-1,i}= S_i(A_2, A_3, \dotsc, A_{n-2})$ is  the  coefficient of $w^{n-1}$ in the expansion of  $G(w)^i$.
\end{proof}
For $B=-1$, the above theorem reduces to  the theorem given in \cite{MR749890}.

The following theorem has been proved in \cite[Theorem 4.4, p.\ 14]{MR3436767} by using the coefficient bounds of the functions in the class $\mathcal{P}$  but we are providing a slightly different proof by making use of the coefficient bounds of the functions in the class $\mathcal{P}[2\beta-1,1]$ $( \beta > 1)$ which shortens the computations involved in the proof to some extent.
\begin{theorem}\label{p3thm4}
Let $f(z)= z+ a_2 z^2  + a_3 z^3 + \cdots \in \mathcal{K }[2 \beta -1,1]$  $(\beta > 1)$ and  $f^{-1}(w)=:F(w)= w + \sum_{n=2}^\infty \gamma_{n} w^n$ in some neighbourhood of the origin.  Then  for $n \geq 2$,
$$ |\gamma_{n}| \leq  \dfrac{1}{n}\prod_{m=0}^{n-2} \left(\dfrac{ 2(\beta -1) + m (2 \beta -1)}{m+1} \right). $$
The result is sharp.
\end{theorem}
\begin{proof}
Since $f \in \mathcal{K }[2 \beta -1,1]$, we have $1 + z f''(z)/f'(z) = p(z)$ where $p(z) =1 + \sum_{i=1}^{\infty} c_i z^i  \in  \mathcal{P}[2\beta-1, 1]$ and  $\beta >1$. This gives
\begin{align}
\dfrac{d}{dw}\left(\dfrac{F(w)}{F'(w)}\right) = 1- \dfrac{F(w) F''(w)}{(F'(w))^2} = p(F(w)) \label{p3eq13.01}
\end{align}
where $ w=f(z)$ lies in some disk around the origin. Integrate the equation \eqref{p3eq13.01} along the line segment  $[0,w]$ and using the power series expansions of $F$ and $p$, we have
\begin{align}
\sum_{n=1}^{\infty} \gamma_{n} w^n = \sum_{n=1}^{\infty} n \gamma_{n} w^n +  \sum_{n=2}^{\infty} \left(\sum_{k=1}^{n-1} k \gamma_{k} \sum_{j=1}^{n-k}  c_j \dfrac{\gamma_{n-k,j}}{n-k+1}\right) w^n \label{p3eq1.114}
\end{align}
 where $\gamma_1=1$ and $\gamma_{n-k,j}$ denotes  the  coefficient of $w^{n-k}$ in the expansion of $F(w)^j$ with  $\gamma_{n-k,n-k}=1$.

 On comparing  the  coefficients of $w^n$, we have
 \begin{align}
 -(n-1) \gamma_{n} =\sum_{k=1}^{n-1}  \dfrac{k \gamma_{k} }{n-k+1}\sum_{j=1}^{n-k} c_j \gamma_{n-k,j} \quad ( n \geq 2). \label{p3eq13.02}
 \end{align}
  Define a function  $f_1: \mathbb{D} \to \mathbb{C}$ such that
\begin{align}
 f_1'(z)= (1- z)^{2(\beta-1)}.  \label{p3eq13.06}
\end{align}  Then  $F_1(w):= f_1^{-1}(w)= w + A_2 z^2 + A_3 z^3 + \cdots$ where for  $n \geq 2$,
\begin{align}
A_n:= \dfrac{1}{n}\prod_{m=0}^{n-2} \left(\dfrac{ 2(\beta -1) + m (2 \beta -1)}{m+1} \right). \label{p3eq13.07}
\end{align}
We shall show that  for all $n \geq 2$, $|\gamma_{n}| \leq A_n$. We  proceed by induction on $n$. Since $p \in \mathcal{P}[2 \beta-1,1]$ $(\beta >1)$,   an application of Lemma \ref{p3lem3} gives  $|c_j| \leq 2(\beta-1)$  for each $j \geq 1$.  Therefore,  the desired estimate holds for $n=2$.
 Assume that the theorem is true for $j \leq n-1$ and thus we have $ |\gamma_{j}| \leq A_j$ for $j \leq n-1$.
Since  $ \gamma_{n,j}= S_j(\gamma_{2}, \gamma_{3}, \dotsc, \gamma_{n-1})$ is a polynomial in $\gamma_{2}$, $\gamma_{3}$, $\dotsc$, $\gamma_{n-1}$ with non-negative coefficients, we have $ |\gamma_{n,j}| \leq  S_j(|\gamma_{2}|, |\gamma_{3}|, \dotsc, |\gamma_{n-1}|) \leq S_j(A_2, A_3, \dotsc, A_{n-1}) $. An application of  induction hypothesis  and bounds of $c_j$ in \eqref{p3eq13.02} gives
\begin{align}
(n-1) |\gamma_{n}| &\leq 2(\beta-1)\sum_{k=1}^{n-1}  \dfrac{k |\gamma_{k}| }{n-k+1}\sum_{j=1}^{n-k}  |\gamma_{n-k,j}| \notag \\
&\leq  2(\beta-1)\sum_{k=1}^{n-1}  \dfrac{k A_{k} }{n-k+1}\sum_{j=1}^{n-k}  S_j(A_2, A_3, \dotsc, A_{n-k-1}) \notag \\
&= 2(\beta-1)\sum_{k=1}^{n-1}  \dfrac{k A_{k} }{n-k+1}\sum_{j=1}^{n-k}   A_{n-k,j} \label{p3eq13.04}
\end{align}
where $A_1=1$ and  $ A_{n-k,j}$ denotes  the  coefficient of $w^{n-k}$ in the expansion of  $F_1(w)^j$ with $ A_{n-k,n-k}=1$.
 We now show that for each $n \geq 2$,
\begin{align}
 2(\beta-1) \sum_{k=1}^{n-1}  \dfrac{k A_{k} }{n-k+1}\sum_{j=1}^{n-k}   A_{n-k,j} = (n-1) A_{n}.\label{p3eq13.05}
 \end{align}

For $f_1$, given by \eqref{p3eq13.06},  we have $$ 1 +  \dfrac{z f_1''(z)}{f_1'(z)} = \dfrac{1- (2\beta-1) z}{1-z}.$$ In terms of  $F_1:=f_1^{-1}$, the above equation can be rewritten as
  $$ \dfrac{d}{dw}\left(\dfrac{F_1(w)}{F_1'(w)}\right) = 1- \dfrac{F_1(w) F_1''(w)}{(F_1'(w))^2} = \dfrac{1- (2 \beta -1)F_1(w)}{1- F_1(w)}.$$
By proceeding as in \eqref{p3eq1.114}, we obtain $$\sum_{n=1}^{\infty} A_{n} w^n = \sum_{n=1}^{\infty} n A_{n} w^n -2(\beta-1) \sum_{n=2}^{\infty} \left(\sum_{k=1}^{n-1} k A_{k} \sum_{j=1}^{n-k} \dfrac{A_{n-k,j}}{n-k+1}\right) w^n.$$
On comparing  the  coefficients of $w^n$, we get  $$(n-1) A_{n}= 2(\beta-1)\sum_{k=1}^{n-1}  \dfrac{k A_{k} }{n-k+1}\sum_{j=1}^{n-k}   A_{n-k,j} \quad (n \geq 2).$$
This proves \eqref{p3eq13.05} and hence, in view of \eqref{p3eq13.04}, we have $|\gamma_n| \leq A_n$. The sharpness follows for the function $f_1$, given in \eqref{p3eq13.06}.
\end{proof}
\begin{corollary}
Let $f \in \mathcal{K}[2 \beta -1,1]$  $(\beta > 1)$ and  $f^{-1}(w)=:F(w)= w + \sum_{n=2}^\infty \gamma_{n} w^n$ in some neighbourhood of the origin.  If $F(w)/F'(w)= w + \sum_{n=2}^\infty \delta_{n} w^n$, then $|\delta_{2}| \leq \beta-1$ and  for $n >2$,
$$|\delta_{n}| \leq  \dfrac{2(\beta-1)}{n(n-1)}\prod_{m=0}^{n-3} \left(\dfrac{2 \beta + m (2 \beta -1)}{m+1} \right).$$
The result is sharp.
\end{corollary}
\begin{proof}
On integrating  the equation \eqref{p3eq13.01} along the line segment  $[0,w]$ and using the power series expansions of $F/F'$, $F$ and $p$, we have
\begin{align}
w+ \sum_{n=2}^{\infty} \delta_{n} w^n = w + \sum_{n=2}^{\infty} \sum_{j=1}^{n-1}  c_j \dfrac{\gamma_{n-1,j}}{n} w^n  \label{p3eq1.1145}
\end{align}
 where  $\gamma_{n-1,j}$ denotes  the  coefficient of $w^{n-1}$ in the expansion of $F(w)^j$ with  $\gamma_{n-1,n-1}=1$.
 Note that $ \gamma_{n-1,j}= S_j(\gamma_{2}, \gamma_{3}, \dotsc, \gamma_{n-2})$ is a polynomial in $\gamma_{2}$, $\gamma_{3}$, $\dotsc$, $\gamma_{n-2}$ with non-negative coefficients.  On comparing the coefficients  of $w^n$ in \eqref{p3eq1.1145} and using Lemma \ref{p3lem3} and  Theorem \ref{p3thm4}, we have
\begin{align}
 |\delta_{n}| &\leq \frac{2(\beta-1)}{n} \sum_{j=1}^{n-1}  S_j(A_2, A_3, \dotsc, A_{n-2}) \notag \\
&= \frac{2(\beta-1)}{n} \sum_{j=1}^{n-1} A_{n-1,j}\label{p3eq13.045}
\end{align}
where  $ A_{n-1,j}=  S_j(A_2, A_3, \dotsc, A_{n-2})$ denotes  the  coefficient of $w^{n-1}$ in the expansion of  $F_1(w)^j$ with $ A_{n-1,n-1}=1$ and $F_1$ is given by \eqref{p3eq13.07}. Corresponding to $F_1$, $F_1(w)/F_1'(w)=w- \sum_{n=2}^{\infty} B_{n} w^n$ where $$B_2:= (\beta-1)\quad \text{and} \quad   B_n:= \dfrac{2(\beta-1)}{n(n-1)}\prod_{m=0}^{n-3} \left(\dfrac{2 \beta + m (2 \beta -1)}{m+1} \right) \, \, (n >2). $$
For $f_1$, given by \eqref{p3eq13.06}, by proceeding as in \eqref{p3eq1.1145}, we have
$$ w- \sum_{n=2}^{\infty}B_{n} w^n = w - \sum_{n=2}^{\infty} \sum_{j=1}^{n-1} 2(\beta-1)\dfrac{A_{n-1,j}}{n} w^n.$$
On comparing the coefficients of $w^n$, we obtain
\begin{align}
B_n= \dfrac{2(\beta-1)}{n}\sum_{j=1}^{n-1} A_{n-1,j}.  \label{p3eq13.0455}
\end{align}
In view of \eqref{p3eq13.045} and \eqref{p3eq13.0455}, the desired estimates follow.
 \end{proof}
In the generalized class $\mathcal{K }[A,B]$ $(-1 \leq B \leq 1 <A)$, the technique used in the  Theorem \ref{p3thm4}  does not hold  true. However,  we are able to give the sharp estimation for the initial  inverse coefficients for functions in $\mathcal{K }[A,B]$.
\begin{theorem}\label{p3thm5}
Let $f(z)= z+ a_2 z^2  + a_3 z^3 + \cdots \in \mathcal{K }[A,B]$  $(  -1 \leq B  \leq 1 <A)$ and  $f^{-1}(w)=:F(w)= w + \sum_{n=2}^\infty \gamma_{n} w^n$ in some neighbourhood of the origin.  Then  for $n= 2, \dotsc,6$,
$$ |\gamma_{n}| \leq  \dfrac{1}{n}\prod_{m=0}^{n-2} \left(\dfrac{ (A -B) + m A}{m+1} \right).$$
The result is sharp.
\end{theorem}
\begin{proof}
Since $f \in \mathcal{K }[A,B]$, $1+ z f''(z)/f'(z) \prec (1+A z)/(1+Bz)$ which is equivalent to $1+ z f''(z)/f'(z) \prec (1-A z)/(1-Bz)$. Let $g(z):= z f'(z)= z + \sum_{n=2}^{\infty}n a_n z^n$ and $p(z):= z g'(z)/g(z)= 1 + b_1 z + b_2 z^2 + \cdots$. Then $p(z) \prec (1- A z)/(1-B z)$ and  for $n >1$,  we have
\begin{align}
(n-1) n a_n = \sum_{k=1}^{n-1} (n-k) b_k a_{n-k}. \label{p3eq14}
\end{align}
It is easy to observe that if $p \prec \varphi$,  then
\begin{align}
p(z)=  \varphi \left(\dfrac{p_1(z)-1}{p_1(z)+1}\right), \quad  p_1(z)= 1+ c_1 z + c_2 z^2 + \cdots \in \mathcal{P}. \label{p3eq15}
\end{align}
Using \eqref{p3eq14} and \eqref{p3eq15} for $\varphi= (1-A z)/(1-B z)$, the coefficients $a_i$ can be expressed  in terms of $c_i$, $A$ and $B$, see \cite{MR3348983}. In particular, we have
\begin{align}
a_2 &= -\dfrac{1}{4}(A-B) c_1, \notag \\
a_3 &= \dfrac{1}{24} (A-B) \left((A-2 B+1)c_1^2  -2 c_2\right),  \notag \\
a_4 &= -\dfrac{1}{192} (A-B) \left((A-2 B+1) (A-3 B+2)c_1^3  - 2(3 A-7 B+4) c_1  c_2 + 8 c_3\right), \notag  \\
a_5 &=  \dfrac{1}{1920}(A-B) \big(-4 (3 A^2-17 A B+11 A+23 B^2-29 B+9)c_1^2 c_2  \notag \\ &\quad{}+ (A-2 B+1) (A-3 B+2) (A-4 B+3) c_1^4  + 16 (2 A-5 B+3)c_1  c_3 \notag \\ &\quad{}+ 12  (A-3 B+2)c_2^2 - 48 c_4\big) \notag \\
\intertext{and}
a_6 &= \dfrac{1}{23040}(A-B)\big( -(A-5 B+4) (A-4 B+3) (A-3 B+2) (A-2 B+1)c_1^5 \notag \\ &\quad{}+ 4  (5 A^3-50 A^2 B+35 A^2+160 A B^2-220 A B+75 A-163 B^3+329 B^2 \notag \\ &\quad{}-219 B+48) c_1^3 c_2  - 16 (5 A^2-30 A B+20 A+43 B^2-56 B+18)c_1^2 c_3 \notag \\ &\quad{}+ 32(5 A-17 B+12) c_2 c_3  -  4 (15 A^2-100 A B+70 A+157 B^2-214 B+72) c_1 c_2^2  \notag \\ &\quad{}+ 48  (5 A-13 B+8) c_1 c_4 - 384 c_5 \big). \notag
\end{align}
Using power series expansions of $f$ and $f^{-1}$ in the relation $f(f^{-1}(w))=w$,  or $$w= f^{-1}(w) + a_2 (f^{-1}(w))^2
 + \cdots, $$ we obtain
 \begin{align}
\gamma_2 &= -a_2, \notag \\
\gamma_3 &= 2 a_2^2-a_3,  \notag \\
\gamma_4 &= -5 a_2^3 + 5 a_2 a_3 - a_4, \notag  \\
\gamma_5 &=  14 a_2^4- 21 a_2^2 a_3 + 6 a_2 a_4+3a_3^2-a_5 \notag \\
\intertext{and}
\gamma_6 &=7 \big(-6 a_2^5 + 12 a_2^3 a_3-4 a_2^2 a_4 + a_2 (a_5-4 a_3^2)+ a_3 a_4\big)-a_6. \notag
\end{align}
Substituting the expressions of $a_i$ in terms of $c_i$ in the above expressions of $\gamma_{i}$, we have
 \begin{align}
\gamma_2 &= \dfrac{1}{4}(A-B) c_1, \notag \\
\gamma_3 &= \dfrac{1}{24} (A-B) \left((2 A-B-1)c_1^2 + 2 c_2\right),  \notag \\
\gamma_4 &= \dfrac{1}{192} (A-B) \left( (2 A-B-1) (3 A-B-2)c_1^3 + 2 (7 A-3 B-4)c_1 c_2 + 8 c_3\right), \notag  \\
\gamma_5 &= \dfrac{1}{1920}(A-B) \big(p(A,B)\ c_1^2 c_2  + (2 A-B-1) (3 A-B-2) (4 A-B-3) c_1^4 \notag \\ &\quad{}+ 8  (11 A-5 B-6)c_1 c_3  + 4  (7 A-B-6)c_2^2 + 48 c_4\big) \notag \\
\intertext{and}
\gamma_6 &= \dfrac{1}{23040}(A-B) \big(q(A,B)\ c_1^2 c_3   + r(A,B) \ c_1 c_2^2 + 384  (2 A-B-1) c_1 c_4  \notag \\ &\quad{}+ s(A,B)\  c_1^3 c_2 +(2 A-B-1) (3 A-B-2) (4 A-B-3) (5 A-B-4)c_1^5   \notag \\ &\quad{}+ 16(25 A-B-24) c_2 c_3  + 384 c_5\big) \notag
\end{align}
where
\begin{align*}
p(A,B)&:= 4  (23 A^2-17 A B-29 A+3 B^2+11 B+9),\\
q(A,B)&:= 8 (101 A^2-81 A B-121 A+16 B^2+49 B+36),\\
r(A,B)&:= 4(127 A^2-58 A B-196 A+3 B^2+52 B+72)
\intertext{and}
s(A,B)&:= 4  (163 A^3-160 A^2 B-329 A^2+50 A B^2+220 A B+219 A\notag \\ &\quad{}-5 B^3-35 B^2 -75 B-48).
\end{align*}
Since $-1 \leq B \leq 1<A$, we can easily see that
\begin{align*}
\dfrac{\partial{p(A,B)}}{\partial{A}}&= 4\big( 29(A-1) + 17(A-B)\big)  > 0, \\
\dfrac{\partial{q(A,B)}}{\partial{A}} &= 8\big( 121(A-1) + 81(A-B)\big)  > 0 \intertext{and}
\dfrac{\partial{r(A,B)}}{\partial{A}}&= 4\big( 196(A-1) + 58(A-B)\big)  >0.
\end{align*}
 Therefore, $p(A,B) > p(1,B) = 12(1-B)^2  \geq 0$; $ q(A,B) > q(1,B) = 128(1-B)^2  \geq 0$ and  $r(A,B) > r(1,B) = 12(1-B)^2 \geq 0$.
 Clearly,
 \begin{align*}
\dfrac{\partial{s(A,B)}}{\partial{A}}&= 4(489 A^2 -658 A - 320 A B + 219 + 220 B + 50 B^2)
 \intertext{and}
\dfrac{\partial^2{s(A,B)}}{\partial{A^2}}&= 4\big(658(A-1) + 320(A-B)\big)  > 0.
\end{align*}
Therefore, $\partial{s(A,B)}/\partial{A}$ is a strictly increasing function of $A$ and hence  $\partial{s(A,B)}/\partial{A}> 200(1-B)^2 \geq 0.$ Consequently, $s(A,B) > s(1,B)= 20(1-B)^3  \geq 0$.

Thus, for $n=2, \dotsc, 6$,  $\gamma_n$  are polynomials in $c_i$ $(i=1, 2, \dotsc, 5)$ with non-negative coefficients. Since $p_1 \in \mathcal{P}$,  $|c_i| \leq 2$ $(i=1, 2, \dotsc)$ and therefore, the maximum of  $|\gamma_n|$ would correspond to $|c_i|=2$. On simplification, we get  the desired estimates.  Define  a function $f_0: \mathbb{D} \to \mathbb{C}$ such that
$$f_0'(z)=\begin{cases}
(1-B z)^{(A-B)/B}, &\text{$B \neq 0$}\\
 e^{-A z}, &\text{$B=0$.}
\end{cases}$$
The result is sharp for the function $f_0$.
\end{proof}


%

\end{document}